\documentclass[a4paper,11pt,reqno]{article}

\usepackage[utf8]{inputenc}
\usepackage[left =2.7cm, right = 2.7 cm]{geometry}

\usepackage{lipsum}
\usepackage{amsfonts}
\usepackage{graphicx}
\usepackage{epstopdf}
\usepackage{algorithmic}
\usepackage{subcaption}
\usepackage{tabularx}
\usepackage{amsmath,amssymb,amsthm}
\usepackage[english]{babel}

\usepackage{xcolor}

\usepackage{tabularx}
\usepackage{booktabs}
\usepackage{makecell}
\usepackage{longtable}
\usepackage[ruled,vlined]{algorithm2e}
\usepackage{algorithmic}
\usepackage{lipsum}

\usepackage{url}

\usepackage{amssymb}
\usepackage{amsfonts}
\usepackage{amsmath}
\usepackage{comment}
\usepackage{tabularx}
\usepackage{booktabs}
\usepackage{makecell}
\usepackage{longtable}
\usepackage{graphicx}
\usepackage{caption}
\usepackage{lipsum}
\usepackage{xcolor}
\usepackage{caption}
\usepackage[hidelinks]{hyperref}
\usepackage{authblk}

\usepackage{enumitem}
\setlist[enumerate]{leftmargin=.5in}
\setlist[itemize]{leftmargin=.5in}

\usepackage{parskip}

\DeclareMathOperator{\spann}{span}

\renewcommand{\a}{\alpha}
\newcommand{\f}{\mathbb{F}}

\newcommand{\K}{\mathbb{K}}

\newcommand{\fqm}{{\mathbb{F}_{q^m}}}

\newcommand{\fq}{{\mathbb{F}_q}}
\newcommand{\fqmst}{\mathbb{F}_{q^m}^*}
\newcolumntype{C}{>{\centering\arraybackslash}X} 

\newtheorem{theorem}{Theorem}

\newtheorem{problem}{Problem}

\newtheorem{proposition}{Proposition}

\theoremstyle{definition}

\newtheorem{definition}{Definition}
\theoremstyle{remark}
\newtheorem{remark}{Remark}

\makeatletter
\def\thm@space@setup{%
  \thm@preskip=\parskip \thm@postskip=0pt
}
\makeatother

\begin{document}
\author{Simran Tinani}
\author{Joachim Rosenthal}

\affil{University of Zürich, Winterthurerstrasse, 8057 Zurich, \\
 WWW home page: 
\texttt{https://www.math.uzh.ch/aa/}}

\title{Existence and Cardinality of $k$-Normal Elements in Finite Fields}

\maketitle           

\begin{abstract}
Normal bases in finite fields constitute a vast topic of large theoretical and practical interest. Recently, $k$-normal elements were introduced as a natural extension of normal elements. The existence and the number of $k$-normal elements in a fixed extension of a finite field are both open problems in full generality, and comprise a promising research avenue. In this paper, we first formulate a general lower bound for the number of $k$-normal elements, assuming that they exist. We further derive a new existence condition for $k$-normal elements using the general factorization of the polynomial $x^m-1$ into cyclotomic polynomials. Finally, we provide an existence condition for normal elements in $\fqm$ with a non-maximal but high multiplicative order in the group of units of the finite field. \end{abstract}
\section{Introduction}

Let $q$ denote a power of a prime $p$, and $\fq$ denote the finite field of order $q$. If $\f$ is an extension field of the field $\K$, we denote by $\mathrm{Gal}(\f/\K)$ the Galois group of the extension field $\f$ over $\K$. We are interested in studying elements in a finite
extension $\fqm$ of degree $m$ over $\fq$. An
element $\a \in \fqm $ is called a normal element over $\fq$ if all its Galois conjugates, i.e. the $m$ elements $\{\alpha, \alpha^q, \ldots, \alpha^{q^{m-1}}\} $, form a basis of $\fqm$ as a vector space over $\fq$. A basis of this form is called a normal basis.

We let $\phi$ denote the usual Euler-phi function for integers. Let $f \in \fq[x]$ be a polynomial with positive degree $m$. Then $\Phi_q(f)$ is defined to be the order of the ring $\left(\frac{\fq[x]}{\langle f \rangle}\right)^{\times},$ where $\langle f\rangle$ denotes the ideal generated by $f$ in $\fq[x]$. In other words, $\Phi_q(f)$ is the number of polynomials co-prime to $f$ and with degree less than $m$. It is well known that normal elements exist in every finite extension $\fqm$ of $\fq$ and that there are precisely $\Phi_q(x^m-1)$ normal elements, and thus $\frac{\Phi_q(x^m-1)}{m}$ normal bases in $\fqm$ \cite[Theorem 2.35, Theorem 3.73]{lidl1997finite}, \cite{hensel1888ueber}, \cite{ore1934contributions}). 

Normal elements are a topic of major significance and interest because they offer an avenue for efficient arithmetic in a finite field $\fq$: for instance, raising an element to the power $q$ is simply a cyclic shift in normal base representation. Normal bases and related concepts such as optimal normal bases and self-dual normal bases find several applications, both theoretical and practical. We refer the interested reader to \cite{ash1989low}, \cite{gao1993normal}, and  \cite{blake1993applications} for more on this topic.

In \cite{huczynska2013existence}, Huczynska et al. introduced the concept of $k$-normal elements as a natural generalization of normal elements. One of the many equivalent ways to define a $k$-normal element $\alpha \in \fqm$ is as an element whose conjugates $\{\a, \a^q, \a^{q^2}, \ldots \a^{q^{m-1}}\}$ span a vector space of dimension $m-k$ over $\fq$. It is then of natural interest to examine the existence and the number of $k$-normal elements. These problems have been shown to be closely tied to the factorization of the polynomial $x^m-1$ \cite{huczynska2013existence}. In this paper, we denote by $n_k$ the number of $k$-normal elements in an extension $\fqm$ of $\fq$. There are numerous known results on bounds on the number $n_0$, several of which build on the lower bounds proved in \cite{frandsen2000density} using properties of the function $\Phi_q$ (see also the improvements on these results in \cite{gao1997density}). For arbitrary $k$, $0< k < m-1$, neither a general rule for the existence of $k$-normal elements nor a general formula for their number $n_k$, when they exist, is known.

Huczynska et al. \cite{huczynska2013existence} have used the approach of Frandsen \cite{frandsen2000density} to give a lower bound on $n_k$ which holds asymptotically, as well as an upper bound which holds in general. However, both their upper and lower bounds depend directly on the number of divisors of $x^m-1$ with degree $m-k$, and are thus difficult to calculate. Moreover, when $x^m-1$ has no divisor of degree $m-k$, the bounds equal zero, which means that the statement about lower bounds does not yield any existence result.
 
In a recent paper, Saygı et al. \cite{saygi2019number} give formulas (in terms of $q$ and $m$) for $n_k$ for cases where $m$ is a power of a prime or of the form $2^v\cdot r$ where $r\neq2$ is a prime and $v \geq 1$, using known results on the explicit factorization of cyclotomic polynomials. In particular, their formulae guarantee existence for certain cases. A recent result by Reis \cite[Theorem 5.5]{reis2018existence} provides a sufficient condition on $m$ for which $k$-normal elements exist for every $0\leq k \leq m$. Some relevant interesting results on the construction of $k$-normal elements, as well as alternate proofs of existing results, are found in \cite{SOZAYACHAN201894}.

In 1987, Lenstra and Schoof \cite{lenstra1987primitive} proved (also see partial proofs by Carlitz \cite{carlitz1952primitive} and Davenport \cite{davenport1968bases}) the Primitive Normal Basis theorem, which states the existence of an element that is simultaneously normal and primitive (i.e. has multiplicative order $q^m-1$ in $\fqmst$). By extension, elements that have high multiplicative orders and also span large subspaces along with their conjugates are of interest. In particular, problems along this line have found mention in \cite{huczynska2013existence}, \cite{kapetanakis2014normal}, \cite{kapetanakisvariations} and \cite{mullen2016some}. The question of the existence of elements in $\fqm$ that are both $1$-normal over $\fq$ and primitive has been answered in entirety in \cite{reis2018existence}, after a partial proof and formulation of the problem in \cite{huczynska2013existence}.

In this paper, we first present a result that guarantees a general lower bound on $n_k$ (for arbitrary $0\leq k \leq m-1$), provided that $k$-normal elements exist. This proves a link between $n_0$ and $n_k$ for $k>0$. Since this result does not make any additional assumption about $k$, $q$ or $m$, it is not derivable from any of the known formulas for $n_k$. 

We further present an existence condition for $k$-normal elements (over $\fq$) in $\fqm$ based on inequalities involving $m$ and $k$. It turns out that under certain constraints on $m$ (loosely put, $m$ must have a sufficiently large common divisor with $q^m-1$), $k$-normal elements exist for $k$ above a minimum lower bound. This result is independent of the factorization of $x^m-1$. Moreover, the conditions on $m$ and $q$ required are weaker than the special forms required in \cite{saygi2019number}, and also cannot be derived from the conditions in \cite[Theorem 5.5]{reis2016existence}. In fact, when $p \nmid m$, our theorem is a generalization of this result. 

Our final contribution is an existence condition for normal elements of multiplicative order $\frac{q^m-1}{q-1}$ in $\fqm$ when $m$ and $q-1$ are co-prime. Using the terminology of \cite{mullen2016some}, this is the same as talking about $0$-normal, $(q-1)$-primitive elements. With this result, we answer a special case of Problem 6.4 posed in \cite{huczynska2013existence}, which deals with high multiplicative order $k$-normal elements in $\fqm$ over $\fq$. Our proof follows the method used by Lenstra and Schoof in proving the Primitive Normal Basis Theorem \cite{lenstra1987primitive}.

\section{Preliminaries}

\begin{definition} An element $\a \in \fqm$ is called $k$-normal if \[\dim_\fq \left(\mathrm{span}_\fq \left\{\alpha, \alpha^q, \ldots, \alpha^{q^{m-1}} \right\} \right) = m-k.\] 
\end{definition}

\begin{remark} It is clear from the definition that an element $\a$ is $0$-normal if and only if it is normal by the usual definition. Also, the only $m$-normal element in $\fqm$ is 0. \end{remark}

Given $\alpha \in \fqm$, we denote by $\mathrm{ord}(\alpha)$ the usual multiplicative order of $\alpha$ in the group $\fqmst$. $\fqm$ may be seen as a module over the ring $\fq[x]$, under the action 
\begin{align} \label{str:1}
\sum_{i=0}^n a_i x^i \cdot \alpha = \sum_{i=0}^n a_i \alpha^{q^i}, \; \alpha \in \fqm.
\end{align}

In other words, the value of the image of $\alpha$ under the action of a polynomial $f(x)=\sum_{i=0}^n a_i x^i$ is the evaluation of $\alpha$ at the $q$-associate \cite[Definition 3.58]{lidl1997finite} of $f(x)$. Note that this is the same as the action of $\fq$-linear maps on $\fqm$. This module structure has been explored in more detail, for instance, in \cite{SOZAYACHAN201894}. Through this module structure, we also have another concept of order, as defined in \cite{lenstra1987primitive} as an additive analogue of the multiplicative order.

\begin{definition}\label{order} Define the function \[\mathrm{Ord}: \fqm \rightarrow \fq[x]\] as follows. For any $\alpha \in \fqm$, $\mathrm{Ord}(\a)$ is the unique monic polynomial generating the annihilator of $\alpha$ under the action defined by Equation \eqref{str:1}, i.e. \[\mathrm{Ann}(\alpha) = \langle\mathrm{Ord}(\alpha)\rangle\text{ in }\fq[x].\]
\end{definition}

We now state an important result which provides several equivalent characterizations of $k$-normal elements.

\begin{theorem} \cite[Theorem 3.2]{huczynska2013existence}\label{equivcond} Let $\alpha$ be an element of $\fqm$ and \[g_\alpha(x):= \sum_{i=0}^{m-1}\alpha^{q^i}\cdot x^{m-1-i} \in \fqm[x].\] Then the following conditions are equivalent: 
\begin{enumerate}
\item $\alpha$ is $k$-normal. 
\item $\gcd(x^m -1, \; g_\alpha(x))$ over $\fqm$ has degree $k$. 
\item $\deg(\mathrm{Ord}(\alpha)) =m - k$. 
\item The matrix $A_\alpha$ defined below has rank $m-k$. \[A_\alpha= \begin{bmatrix}
\alpha & \alpha^q & \alpha^{q^2} & \cdots & \alpha^{q^{m-1}} \\
\alpha^{q^{m-1}} & \alpha & \alpha^{q} & \cdots & \alpha^{q^{m-2}} \\
\vdots & \vdots & \cdots & \vdots & \vdots\\
\alpha^q & \alpha^{q^2} & \alpha^{q^3} & \cdots & \alpha \\
\end{bmatrix}. 
\] 
\end{enumerate}
\end{theorem}



The following result on the number of $k$-normal elements will also prove useful.

\begin{theorem} \cite[Theorem 3.5]{huczynska2013existence}
The number of $k$-normal elements of $\fqm$ over $\fq$ equals 0 if there is no $h \in \fq[x]$ of degree
$m - k$ dividing $x^m-1$; otherwise it is given by
\begin{equation}\label{number_knormal}
\sum_{\substack{h \mid x^m-1 \\ \deg(h)=m-k}} \Phi_q(h), 
\end{equation}
where divisors are monic and polynomial division is over $\fq$.
\end{theorem}

It is known that $x^m-1$ factorizes over $\fq$ into the product of cyclotomic polynomials of degrees dividing $m$ \cite[Theorem 2.45]{lidl1997finite}. Moreover, for $p \nmid d$ (recall that $p$ is defined as $p=\mathrm{char}(\fq)$), each of the irreducible factors of the cyclotomic polynomial $Q_d(x)$ has degree $\frac{\phi(d)}{r}$, where $r$ is the multiplicative order of $d \mod q$ \cite[Theorem 2.47]{lidl1997finite}. Since there is no known closed formula for this number, there is also no closed-form complete factorization (i.e. factorization into irreducibles) of $x^m-1$ over $\fq$. Thus, the above theorem does not give direct answers about the existence of $k$-normal elements for $k>0$. However, it may be used to ascertain the existence of $k$-normal elements for certain values of $k$. In the next two sections, we look at some interesting results on $k$-normal elements which can be derived in certain special cases using
Thereom \ref{number_knormal}.

\section{Number of $k$-Normal Elements}

For $k=0$, the formula in Theorem \ref{number_knormal} yields the well-known value $\Phi_q(m)$ for the number of normal elements over $\fq$ in $\fqm$ \cite[Theorem 3.37]{lidl1997finite}. Since $x^m-1$ always has the divisor $x-1$ of degree 1 and hence also a divisor of degree $m-1$ (and since $\Phi_q(f(x)) \neq 0$ for any nonzero polynomial $f(x)$), we always have 1-normal and $(m-1)$-normal elements in $\fqm$. It has been observed in \cite{huczynska2013existence} that the only values of $k$ for which $k$-normal elements are guaranteed to exist for every pair $(q, \;m)$ are 0, 1 and $m-1$. In fact, as noted in \cite{reis2018existence}, if $q$ is a primitive root modulo $m$, $\frac{x^m-1}{x-1}$ is irreducible and so for $1<k<m-1$, $k$-normal elements do not exist.

In certain other cases, it is possible to use information about the factorization of $x^m-1$ along with Theorem \ref{number_knormal} to gain insights into the number of $k$-normal elements for different values of $k$. In \cite{saygi2019number}, the authors provide explicit formulas for $k$-normal elements for degrees $m$ that are either prime powers or numbers of the form $2^v\cdot r$, for a prime $r\neq 2$, under certain other constraints on $q$ and $m$. Below we state one of their noteworthy results.


\begin{proposition}[{{\cite[Proposition 1]{saygi2019number}}}] \label{mpowerofchar} Let $\mathrm{char}(\fq) = p$ and $m = p^r$ for some positive integer $r$. Then the number of $k$-normal
elements of $\fqm$ over $\fq$ is given by \[ (q - 1)\cdot q^{m-k-1},\] where $k = 0, 1, \ldots , m - 1$. \end{proposition}


The following result by Huczynska et al. \cite{huczynska2013existence} formulates a lower bound for the number of $k$-normal elements when the extension degree $m$ is large enough. 


\begin{theorem}[{{\cite[Theorem 4.6]{huczynska2013existence}}}] \label{asymexistence} Let $c_{m-k}$ denote the number of divisors of $x^m-1$ with degree $m-k$. There is a constant $c$ such that for all $q\geq 2$ and $m>q^c$, the number of $k$-normal elements of $\fqm$ over $\fq$ is at least \[ 0.28477\cdot q^{m-k}\cdot \frac{c_{m-k}}{\sqrt{\log_q(m)}}. \] \end{theorem}

Note that there is no simple rule or formula for the value $c_{m-k}$ in terms of $m$, $k$ and $q$, and it may equal zero. So, the above result does not yield an existence condition.

We now proceed to build a general result on the number of $k$-normal elements, assuming that they exist. For this purpose, we consider the structure of $\fqm$ as an $\fq[x]$-module under the action defined by Equation \eqref{str:1}. We follow the approach in \cite{hyde2018normal}, which is based on the observation that for $\K = \fqm$ and $G=\mathrm{Gal}(\K/\fq)$, the group of invertible elements $\K[G]^{\times}$ of the group algebra $\K[G]$ acts on the set of normal elements of $\fqm$. Using this group action, the author of \cite{hyde2018normal} formulates an alternative method to count normal elements. We adapt the same argument to find a lower bound on the number of $k$-normal elements in $\fqm$ when they exist. 
 
 \begin{theorem}\label{lowerbound} Let $k \in \{0, 1, \ldots, m\}$ and let $n_k$ denote the number of $k$-normal elements in $\fqm$. If $n_k>0$, i.e. if $k$-normal elements exist in $\fqm$, then \[n_k \geq \frac{\Phi_q(x^m-1)}{q^k}.\]
 \end{theorem}
 
\begin{proof}
Denote $G:=\mathrm{Gal}(\fqm/\fq)$ and $\K:=\fq$. Let $S_k$ be the set of $k$-normal elements over $\fq$ in $\fqm,$ and assume that $S_k \neq \emptyset$. Let $\K[G]^{\times}$ be the group of invertible elements of the group algebra $\K[G]$. The map 
\begin{align} \label{axnknorm}  \nonumber \K[G]^{\times} \times S_k  &\rightarrow S_k, \text{ given by}\\ 
 \left(\sum_{h \in G} a_h \cdot h\right) \cdot \alpha &= \sum_{h \in G} a_h \cdot (h\cdot \alpha) 
 \end{align} for $\alpha \in \fqm$ and coefficients $a_h \in \K$ defines a group action. The rest of the axioms are clear, and only thing that needs to be verified is that $k$-normal elements map to $k$-normal elements. To see this, note that an element $\psi=\sum_{h \in G} a_h \cdot h$ of $\K[G]^{\times}$ is a field automorphism of $\fqm$, and so the images of subspaces of dimension $m-k$ also have dimension $m-k$. So, for a $k$-normal element $\alpha$, \begin{align*} \dim (\spann\{\psi(\alpha), \psi(\alpha^q), \ldots, \psi(\alpha^{q^{m-1}})\}) &=  \dim (\spann\{\psi(\alpha), \psi(\alpha)^q, \ldots, \psi(\alpha)^{q^{m-1}}\}) \\ &=  \dim (\spann\{\alpha, \alpha^q, \ldots, \alpha^{q^{m-1}}\}) = m-k. \end{align*}

 Now note that for a generator $\sigma$ of $G$ we have a ring isomorphism \begin{align}
 \left(\frac{\fq[x]}{\langle x^m-1\rangle}\right)&\mapsto \K[G]  \nonumber \\
x&\mapsto\sigma. \label{gpisom1} 
\end{align} Therefore, \begin{equation}
\K[G]^{\times} \cong \left(\frac{\fq[x]}{\langle x^m-1\rangle}\right)^{\times} \hspace{4mm} \text{ (as groups)}. \label{gpisom2}
\end{equation}

We conclude that through the isomorphism \eqref{gpisom1} the group action \eqref{axnknorm} induces a group action  \begin{align}  & \left(\frac{\fq[x]}{\langle x^m-1\rangle}\right)^{\times}\times \;  S_k \mapsto S_k \nonumber \\ \text{ given by } &
 \left(\sum_{i=0}^{m-1} f_i \cdot x^i\right) \cdot  \alpha = \sum_{i=0}^{m-1} f_i \cdot \sigma^i(\alpha) = \sum_{i=0}^{m-1} f_i \cdot \alpha^{q^i}. \label{gpaxnused}
 \end{align}

Denote $H:= \left(\frac{\fq[x]}{\left(x^m-1\right)}\right)^{\times}$. For any $k$-normal element $\alpha$, we have 

\begin{align}
\nonumber \mathrm{Stab}(\alpha)&= \{p(x)\in H : p(x)\cdot \alpha = \alpha\} \\ 
\nonumber &= \{p(x) \in 
H : (p(x)-1)\cdot \alpha=0 \}\\
\label{stabilizer} &= \{p(x) \in H : \mathrm{Ord}(\alpha) \text{ divides } (p(x)-1). \}
\end{align}

We know from Theorem \ref{equivcond} that $\mathrm{Ord}(\alpha)$ is a polynomial of degree $m-k$. Equation \eqref{stabilizer} implies that for $p(x) \in H$, \begin{equation}\label{stabcondition} p(x) \in \mathrm{Stab}(\alpha) \iff p(x) = \mathrm{Ord}(\alpha) \cdot r(x) + 1, \text{ with } \deg(r(x))\leq k-1.\end{equation}

Hence, the number of possible distinct values for $p(x) \in \mathrm{Stab}(\alpha)$ cannot exceed the number of polynomials with degree less than $k$. More precisely, \begin{equation} \label{staborder}
\left|\mathrm{Stab}(\alpha)\right| \leq \min(|H|, \; q^k) = \min\left(\Phi_q(x^m-1), \; q^k\right) \leq q^k. 
\end{equation} 

Finally, Equation \eqref{staborder} and the Orbit-Stabilizer Theorem together give \[\left|\mathrm{Orb}(\alpha)\right| = \left|\frac{H}{\mathrm{Stab}(\alpha)}\right| \geq \frac{\Phi_q(x^m-1)}{q^k}. \]

Since the action \eqref{gpaxnused} is on $k$-normal elements, it is now clear that the number $n_k$ of $k$-normal elements satisfies $n_k \geq \frac{\Phi_q(x^m-1)}{q^k} $, thus completing the proof.
\qed
 \end{proof}
 
\begin{remark}\label{remarkref}
Note that if a $k$-normal element $\alpha$ exists, then the lower bound in Theorem \ref{number_knormal} is, in fact, for the number of $k$-normal elements lying in a single orbit, and therefore in $\spann_\fq\{\alpha, \alpha^q, \alpha^{q^2}, \ldots, \alpha^{q^{m-1}}\}$.
\end{remark}

\begin{remark} In \cite{hyde2018normal}, it is shown that for the case of normal elements (i.e. $k=0$), the action \eqref{gpaxnused} is both free (i.e. $u\cdot \alpha=\alpha \implies u=1$) and transitive. This yields an alternate proof of the well-known result that the number of normal elements in $\fqm$ is equal to  $\Phi_q(x^m-1)$. For $k>0$ it is clear that for every $k$-normal $\alpha$, there exists $u \in \K[G]$ such that $u\cdot \alpha = \alpha$. However, it is unclear whether such a $u$ can be found in $\K[G]^\times$ or if the action is transitive. So, we cannot directly adapt the argument as in \cite{hyde2018normal} to count the exact number of $k$-normal elements. However, as shown by the above theorem, the action may nevertheless be used to obtain a lower bound. \end{remark}

\section{Existence of $k$-Normal Elements}

From the previous section, it is clear that some results on the number of $k$-normal elements automatically imply their existence. For instance, the existence of $k$-normal in $\fqm$ for $m$ a power of the characteristic $p$ is established as an immediate corollary of Proposition \ref{mpowerofchar}. On the other hand, the cardinality formula in Theorem \ref{asymexistence} gives the value zero when $x^m-1$ has no divisor with degree $m-k$, and thus yields no condition for the existence of $k$-normal elements. Similarly, the statement on cardinalities in Theorem \ref{lowerbound} holds only under the assumption that $k$-normal elements exist in $\fqm$. We now shift our focus to finding existence conditions for $k$-normal elements over $\fq$. We begin by presenting (a slight rewording of) a result by Reis, which is closely related to our existence result.

\begin{theorem}[\cite{reis2016existence}]\label{similar1} Let $q$ be a power of a prime $p$ and let $m \geq 2$ be a positive integer such that every prime divisor of $m$ divides $p\cdot (q - 1)$. Then $k$-normal elements exist for all $k= 0,1, 2, \ldots, m$. \end{theorem}

Clearly, we get the existence implication of Proposition \ref{mpowerofchar} as a corollary of the above theorem. Although this theorem significantly extends Proposition \ref{mpowerofchar}, it still restricts the prime factorization of $m$ to be of a particular form, and thus limits the allowed values of $m$. It is easy to see that it does not apply to simple examples like $q=5$, $m=6$, and $q=8$, $m=6$, where $k$-normal elements are known to exist for every $k= 0,1, 2, \ldots, m$. We now state the main result of this section, a sufficient condition for the existence of $k$-normal elements, which does not, unlike Proposition \ref{mpowerofchar} and Theorem \ref{similar1}, require $m$ or its prime factors to be of a fixed type. This result is also independent of the factorization of $x^m-1$ into irreducibles over $\fq$, and is derived using only the general factorization into cyclotomic polynomials. Before the main theorem, we prove a number theoretic result which will be used. The proof of the below proposition was inspired by the proof of Theorem 6.3 in \cite{luneburg2012translation}.

\begin{proposition}\label{numth}
Let $a$ and $m$ be arbitrary natural numbers and suppose that $m \nmid a^m-1$. Then $m$ has a prime factor that does not divide $a^m-1$. \end{proposition}

\begin{proof}
We proceed by contradiction. Suppose that the statement is false and let $p$ be any prime divisor of $m$. By hypothesis, $p \mid a^m-1$. Write $m = p^b \cdot s$, with $b \geq 1$ and $p \nmid s$. We have \begin{align}
    0 &= a^m-1 \mod p \nonumber \\
    &= (a^{sp^b}-1) \mod p \nonumber  \\
    &= (a^{s}-1)^{p^b} \mod p \nonumber \\
    \implies a^s &= 1 \mod p. \label{indxnhyp}  
\end{align}
We claim that $a^m-1 = 0 \mod p^b$, or in other words, $a^{sp^b}-1 = 0 \mod p^b$. We prove this claim by induction on $b$. 

For $b=1$, the statement $a^{ps}-1 = 0 \mod p$ is true by the hypothesis of the proposition. Now assume that $a^{sp^b}-1 = 0 \mod p^b$ for some $b \geq 1$. Then,
\begin{align}
    a^{sp^{b+1}}-1 &= {(a^{sp^b})}^p-1 \nonumber \\
    &= ({a^{sp^b}}-1)(1 + a^{sp^b} + a^{2sp^b} + \ldots + a^{(p-1)sp^b}) \label{ntheq}.
\end{align}
By the induction hypothesis, $p^b \mid a^{sp^b}-1$. Also, from \eqref{indxnhyp}, we have \begin{align*}
    a^s &= 1 \mod p \\
    \implies a^{isp^b} &=1 \mod p \; \forall \; 0\leq i \leq p-1. \\
    \implies 1 + a^{sp^b} + a^{2sp^b} &+ \ldots + a^{(p-1)sp^b}  = 0 \mod p
\end{align*}

Combining these results, \eqref{ntheq} clearly gives $p^{b+1}\mid a^{sp^{b+1}}-1$, thus proving the result for $b+1$. By induction, the result holds for every $b\geq 1$, and therefore for every $m=sp^b$. So, we may now conclude that $a^m -1 = 0 \mod p^b$ for $m$, $b$, $p$ as in the proposition. Since this holds for any prime factor $p$ of $m$, this implies that $m \mid a^m -1$, which is a contradiction to the assumption. Hence, we must have $p \nmid a^m-1$ for some prime divisor $p$ of $m$. The proof is now complete. \qed

\end{proof}

\begin{remark}
Note that if $p \nmid m$ and the hypothesis of Theorem \ref{similar1} by Reis holds, i.e. every prime factor of $m$ divides $p \cdot (q-1)$ then Proposition \ref{numth} says that we are in the case $m \mid q^m-1$. In this case it will become clear that our theorem is a generalization of the result of \ref{similar1}.
\end{remark}

\begin{theorem}\label{existencecondition} If $m\mid (q^m-1)$, then $k$-normal elements exist in $\fqm$ for every integer $k$ in the interval $0\leq k \leq m-1$. If $m \nmid q^m-1$, let $d=\gcd(q^m-1, \; m)$. Assume that $\sqrt{m} < d$. Let $b$ denote the largest prime divisor of $m$ that is a non-divisor of $q^m-1$ ($b$ exists by Proposition \ref{numth}). Then, for $k \geq m - d - b+1$, $k$-normal elements exist in $\fqm$. In particular, if $m$ is a prime not dividing $q^m-1$, then we have $b=m$, $d=1$, and so $k$-normal elements exist for every $k$ in the interval $0\leq k \leq m-1$. \end{theorem}

\begin{proof} We know from Equation \eqref{number_knormal} that the number of $k$-normal elements in $\fqm$ is given by \[\sum_{\substack {h \mid x^m - 1 \\ \deg h = m-k}} \Phi_q(h(x)).\] Thus, normal elements exist in $\fqm$ if and only if $x^m-1$ has a divisor of degree $m-k$. First note that for $d=\gcd(q^m-1, m)$, we have $d \mid q^m-1$, the order of $\fqmst$, so by the general properties of a finite cyclic group, there are precisely $d$ elements $\alpha$ in the group $\fqmst$ satisfying $\alpha^d=1$, and so $d$ elements must also satisfy $\alpha^m=1$. Thus, $x^m-1$ has precisely $d$ linear factors over $\fqm$. Let its roots in $\fqm$ be $\alpha_1, \alpha_2, \ldots, \alpha_d$.

If $m \mid q^m-1$, then $d=m$, and $x^m-1$ splits into linear factors over $\fqm$. Thus, in this case, for any $k \in \{0, 1, 2, \ldots, m-1\}$, one may always combine $m-k$ of the $m$ linear factors to obtain a factor of degree $m-k$ of $x^m-1$. Hence, we are done for this case. Note that the same conclusion could have been drawn by directly applying Theorem \ref{number_knormal} and using the fact that the polynomial splits into linear factors.

If $m \nmid q^m-1$, then $d<m$. Assume that for some $k \in \{0, 1, 2, \ldots, m-1\}$, no $k$-normal element exists in $\fqm$. It is known that $x^m -1$ has the following factorization over $\fq$: \[x^m -1 = \prod_{t \mid m} Q_t(x).\] where $Q_t(x)$ denotes the $t^{th}$ cyclotomic polynomial, and is known to have coefficients in $\fq$ \cite[Theorem 2.45]{lidl1997finite}. Write \begin{align*}
x^m - 1 &= \prod_{t \hspace{1mm}\mid \hspace{1mm} d} Q_t(x)\cdot \prod_{\substack{t \hspace{1mm} \mid \hspace{1mm} m \\ t \hspace{1mm}\nmid \hspace{1mm} q^m-1}} Q_t(x) \\ 
&= (x^d-1)\cdot \prod_{\substack{t \hspace{1mm}\mid \hspace{1mm} m \\ t \hspace{1mm}\nmid \hspace{1mm}q^m-1}} Q_t(x) \\
&= (x-\alpha_1)\cdot (x-\alpha_2)\cdot \ldots \cdot (x-\alpha_d) \cdot \prod_{\substack{t \hspace{1mm}\mid \hspace{1mm} m \\ t \hspace{1mm}\nmid \hspace{1mm}q^m-1}} Q_t(x),
\end{align*}
 where the last step follows from the fact that $d \mid q^m-1$, so as in the first case, $x^d-1$ splits in $\fqm$. Now, let $b$ be the largest prime dividing $m$ but not $q^m-1$ (such a prime exists by Proposition \ref{numth}). Then $Q_b(x)$ figures in the latter product of the above equation. Since no $k$-normal element exists in $\fqm$, $m-k$ must be greater than the number $d$ of linear factors, and it must be impossible to combine the factors of degree greater than 1, in particular, $Q_b(x)$, with the linear factors to obtain a factor of degree $m-k$. Mathematically, we get, after minor rearrangement, \begin{align}
& k< m-d, \label{cond1} \\ &\text{and } \nonumber  \\
& \text{either } k > m - \phi(b) \text{ or } k < m-d-\phi(b). \label{cond2}  
\end{align}

Now, since $b$ is a prime dividing $m$ but not $q-1$, $b$ must divide $\frac{m}{d}$. In particular, $b \leq \frac{m}{d}$. From the hypothesis $\sqrt{m}< d$, we get $b \leq \frac{m}{d} < d$, and so \begin{align} m-\phi(b) &= m - b + 1  \nonumber \\  &> m-d+1 > m-d \nonumber \\  &>k, \label{finalcondn} \end{align}
where the last step follows from Eq. \eqref{cond1}. We now immediately note that the former condition in Eq. \eqref{cond2} is incompatible with Eq. \eqref{finalcondn}, and so it cannot hold. Therefore, the latter condition of Eq. \eqref{cond2} must be satisfied, i.e. we must have \[k < m-d-\phi(b) = m-d-b+1\]
 for $k$ such that $k$-normal elements do not exist. 
Hence, we conclude that for all $k \geq m - d - b + 1$, $k$-normal elements exist in $\fqm$, as required. 

Finally, it is clear that if $m$ is a prime, then we have $b=m$, $d=1$, and so $k$-normal elements exist for every $k$ in the interval $0\leq k \leq m-1$ by the above condition. 
\qed
\end{proof}

\begin{remark} If $m$ is composite and does not divide $q^m-1$, then we cannot conclude the existence of $k$-normal elements for every value of $k$ using the above theorem. This follows from the following argument, which was provided by one of the reviewers of this paper. Since $b$ and $d$ are different divisors of $m$, then $b+d\leq \frac{m}{2}+\frac{m}{3}$, which is incompatible with the condition $m \leq d + b - 1$.  \end{remark}

\begin{remark} Note that the fact that $b$ is a prime plays a key role in the above proof. If $b$ is, instead, an arbitrary divisor of $m$ that does not divide $q-1$, then it is not guaranteed that $b$ divides $\frac{m}{d}$ (E.g. consider $q=25, m=20, b=10$). So the argument may not hold true even though the inequality $\frac{m}{d}<d$ may hold. \end{remark}

We now reconsider the two examples considered before.  For $q=5$, $m=6$, we have $q^m-1 = 15624$, which is divisible by 6. So, Theorem \ref{existencecondition} shows that $k$-normal elements exist in $\fqm$ for every $k\in \{0, 1, \ldots, m\}$. For $q=8$, $m=6$, we have $q^m-1 = 262143$, and so $d=\gcd(q^m-1, \; m) = 3 > \sqrt{6}$. The largest prime $b$ that divides $6$ and not $262143$ is clearly 2. So, Theorem \ref{existencecondition} shows that $k$-normal elements exist in $\fqm$ for every $k\geq m-d-b+1$, i.e. for every $k\geq 2$. Since we know that $0$- and $1$-normal elements always exist in $\fqm$, we conclude that in this case $k$-normal elements exist for every $k\in \{0, 1, \ldots, m\}$. The exact numbers for these two examples are listed in Tables 4 and 2, respectively, in Section 6.


\section{Normal Elements with Large Multiplicative Order}

So far, we have studied the ``additive" structure of $\fqm$ as a vector space over $\fq$. It is also of interest to study the relation between this additive structure and the natural multiplicative structure of $\fqmst$. One of the most noteworthy results in this direction is the Primitive Normal Basis Theorem (\cite{lenstra1987primitive}, \cite{carlitz1952primitive}, \cite{davenport1968bases}). We state some of its proposed generalizations of this result in Section 6. Below, we state and prove an existence result for normal elements (i.e. $k=0$) with multiplicative order $\frac{q^m-1}{q-1}$ in $\fqm$. It turns out that such elements always exist if $m$ and $q-1$ are co-prime, and that this may be derived using the same methods as Lenstra and Schoof \cite{lenstra1987primitive} in the proof of the Primitive Normal Basis Theorem.

\begin{theorem}\label{existencecondition2}
Suppose that $(m,q-1)=1$. Then $\fqm$ has a normal element with multiplicative order $\frac{q^m-1}{q-1}$.
\end{theorem}
\begin{proof}

Let $k:=\frac{q^m-1}{q-1}$. Define \begin{align*} A= & \{\alpha \in \fqm : \mathrm{Ord}(\alpha) = x^m-1\}, \\
B= & \{\alpha \in \fqmst : \mathrm{ord}(\alpha) = k\}, \\
C= & \{\alpha \in \fqm : \alpha^{(q-1)^2}=1\}, 
 \end{align*}  where the sets $A$ and $C$ are defined identically as in the proof of Lenstra and Schoof, and $B$ is defined as the set of elements with order $k$, rather than primitive elements. Note that $C$ is a subgroup of $\fqmst$. Also note that since the definitions of $A$ and $C$ are unchanged, we may use directly the result (1.12) of the original proof in \cite{lenstra1987primitive}. We state this as follows. For the set $CA$ defined as \[CA = \{\gamma \cdot \a : \gamma \in C, \a \in A\},\] we have \begin{equation} \label{ca=a}
CA=A.
\end{equation} 

Let $BC$ denote the set $BC = \{\beta \cdot \gamma : \beta \in B, \gamma \in C\}$. Now, since Equation \eqref{ca=a} holds, the exact same argument as in the original proof also yields the result indexed (1.13) in \cite{lenstra1987primitive}. Since we have a different $B$, we prove it below. The proof is identical for $B$ defined as the set of elements of any multiplicative order.

If $\a \in A$, $\beta \in B$, $\gamma \in C$ are such that $\alpha = \beta \cdot \gamma \in B\cdot C$, then $\beta = \a \cdot \gamma^{-1} \in CA \cap B = A \cap B$,
and so we have  \begin{equation} \label{emptycondition} A \cap B = \emptyset \iff A \cap BC = \emptyset.\end{equation}  

As in the original paper, we use Equation \eqref{emptycondition} and prove that $A \cap B\cdot C \neq \emptyset$ to conclude that $A\cap B \neq \emptyset$. 

Let $H$ denote the unique subgroup of order $k$ in $\fqmst$.
Here, \begin{align*} BC &= \left\{\beta\cdot \gamma : \beta \in B, \: \gamma \in C \right\} \\
&= \left\{\beta \cdot \gamma : \beta \text{ generates } H, \: \gamma \in C \right \} \\
&= \left\{\beta \cdot \gamma : \beta\cdot C \text{ generates } \frac{H}{C}, \:  \gamma \in C \right\} \\
&= \left\{\beta \cdot \gamma : \beta \cdot C\cap H \text{ generates } \frac{H}{H \cap C}, \: \gamma \in C \right\}.
\end{align*}

Now note that  \begin{align*} \gcd(k,(q-1)) &= \gcd\left(\frac{q^m-1}{q-1}, \; q-1\right) \\ & = \gcd\left(1 + q + q^2 + \ldots + q^{m-1}, \; q-1\right) \\& = \gcd(m, \; q-1) \\&=1, \end{align*}where the second last equality can be checked by direct computation for general values of $m$ and $q$, and the last equality follows by the hypothesis of the theorem. We now have $|C| = (q-1)\cdot \gcd(q-1, \; m) =(q-1)$. So, in this case, $C$ is the unique subgroup of $\fqmst$ with order $q-1$. Thus, $C$ and $H$ are subgroups with co-prime orders, and therefore intersect trivially. Now let \[D=\{\alpha \in \fqmst : ord(\alpha) = q^m-1\}\] denote the set of generators of $\fqmst$. We claim that \[D \subseteq BC.\]

To see this, pick $\alpha \in D$. Since $\gcd(k, \; q-1)=1$, there exist integers $a$ and $b$ such that \[a\cdot k + b\cdot (q-1) = 1.\] This implies that $(a, q-1)=1$ and $(b,k)=1$. Thus, $\alpha^{ka}$ has order $q-1$ and $\alpha^{b (q-1)}$ has order $k$.

Thus, $\alpha = \alpha^{b (q-1)} \cdot  \alpha^{k a}$, with $\alpha^{b (q-1)} \in B$ and $\alpha^{k a} \in C$. We have hereby proved that $D \subseteq BC$. We now have $A \cap D \subseteq A \cap BC$. But, by \cite[result (1.10)]{lenstra1987primitive}, we have $A \cap D \neq \emptyset$, and so we must also have $A \cap BC \neq \emptyset$. By Equation \eqref{emptycondition}, we conclude that $A \cap B \neq \emptyset$.

Hence, $\fqm$ contains a normal element with multiplicative order $k=\frac{q^m-1}{q-1}$, as required.
\qed \end{proof} 

\section{Examples} 

We now demonstrate Theorems \ref{lowerbound}, \ref{existencecondition}, and \ref{existencecondition2} by providing concrete examples. The following cardinalities were derived by an exhaustive search using the
algebra software package SageMath \cite{sagemath}. Each table below corresponds to the extension $\fqm$ of $\fq$, and shows that the number of $k$-normal elements, whenever nonzero, is greater than or equal to the number $\frac{\Phi_q(x^m-1)}{q^k}$ (which has been rounded off to two decimal places in the table), as stated in Theorem \ref{lowerbound}. Below each table, we give the number of normal elements with multiplicative order $\frac{q^m-1}{q-1}$.  In the terminology of \cite{mullen2016some}, we call these $(q-1)$-primitive normal elements. Clearly, Theorem \ref{existencecondition2} is validated by the fact that all these numbers are non-zero. 

We have already discussed Tables 4 and 6 in the light of Theorems \ref{similar1} and \ref{existencecondition}. On the other hand, note that for the example in Table 5, Theorem \ref{similar1} is applicable, while Theorem \ref{existencecondition} is not. As we have noted before, this happens precisely when $p \div m$ and the hypothesis of $\ref{similar1}$ holds. This shows that neither of these two results is stronger than the other. In the case of Table 8, the assumptions of both theorems hold and both guarantee the existence of $k$-normal elements for every value of $k$ less than $m$.  For Tables 1, 2, 3, and 7, neither Theorem \ref{similar1} nor Theorem \ref{existencecondition} applies. In fact, Table 3 shows that 3-normal elements and 7-normal elements over $\f_2$ do not exist in $\f_{1024}$. 

\begin{longtable}{c c}
\begin{minipage}[t][3cm]{.5\linewidth}
\begin{tabular}{@{}|c|c|c|@{}}
\multicolumn{3}{c}{{{\normalsize Table 1: $\f_{8}/\f_2$ ($q=2$, $m=3$)}}}
\vspace{0.9mm} \\ 
\toprule
k & \# of $k$-normal elements & $\displaystyle \dfrac{\Phi_q(x^m-1)}{q^k}$  \\ \midrule
0 & 4 & 4  \\
1 & 4 & 4 \\
2 & 2 & 2  \\
3 & 1 & 1  \\ \bottomrule
\end{tabular}\\
\vspace{0.5mm}

\centering
\# of $(q-1)$-primitive normal \\ elements = 4
\end{minipage} \vspace{8mm} \quad \quad \vspace{8mm}
    &
\begin{minipage}[t][3cm]{.5\linewidth}
\begin{tabular}{@{}|c|c|c|@{}}
\multicolumn{3}{c}{{{\normalsize Table 2: $\f_{59049}/\f_9$ ($q=9$, $m=5$)}}}
\vspace{0.9mm} \\ 
\toprule
k & \# of $k$-normal elements & $\displaystyle \dfrac{\Phi_q(x^m-1)}{q^k}$ \\ \midrule
 0   & 51200    &51200 \\
 1  & 6400  & 5688.89\\
 2& 1280 & 632.10 \\
 3 &160& 70.23 \\
 4 & 8 & 7.80\\ \bottomrule
\end{tabular} \\

\vspace{0.5mm}
\centering
\# of $(q-1)$-primitive normal \\ elements =  5750
\end{minipage}\quad \quad \vspace{9mm}
\\
\begin{minipage}[t][6cm]{.5\linewidth}
\begin{tabular}{@{}|c|c|c|@{}}
\multicolumn{3}{c}{{{\normalsize Table 3: $\f_{1024}/\f_2$ ($q=2$, $m=10$)}}} \vspace{0.9mm} \\ 
\toprule
k & \# of $k$-normal elements & $\displaystyle \dfrac{\Phi_q(x^m-1)}{q^k}$  \\ \midrule
0   & 480    &480 \\
 1  & 240  &240 \\
 2& 240 & 120 \\
 3 &0& 60 \\
 4 & 35 & 30 \\
 5 & 15 & 15 \\
 6 &15 & 7.5 \\
 7 & 0 & 3.75 \\
 8 & 2 & 1.875 \\
 9 &1 & 0.94 \\ \bottomrule
\end{tabular} \\

\centering
\# of $(q-1)$-primitive normal \\ elements = 290
\end{minipage} \quad \quad \vspace{9mm}
    &
\begin{minipage}[t][6cm]{.5\linewidth}
\begin{tabular}{@{}|c|c|c|@{}}
\multicolumn{3}{c}{{{\normalsize Table 4: $\f_{262144}/\f_8$ ($q=8$, $m=6$)}}} \vspace{0.9mm} \\ 
\toprule
k & \# of $k$-normal elements & $\displaystyle \dfrac{\Phi_q(x^m-1)}{q^k}$ \\ \midrule
 0   & 225792    &225792 \\
 1  & 28224  &  28224 \\
 2& 7560 & 3528\\
 3 &441 &441 \\
 4 & 119 & 55.13\\
 5 & 7 & 6.89 \\ \bottomrule
\end{tabular} \\

\centering
\# of $(q-1)$-primitive normal \\ elements =  20124

\end{minipage}
    \\
\begin{minipage}[t][5cm]{.5\linewidth}
    \vspace{0.5mm}
\begin{tabular}{@{}|c|c|c|@{}}
\multicolumn{3}{c}{{{\normalsize Table 5: $\f_{729}/\f_3$ ($q=3$, $m=6$)}}} \vspace{0.9mm} \\ 
\toprule
k & \# of $k$-normal elements & $\displaystyle \dfrac{\Phi_q(x^m-1)}{q^k}$  \\ \midrule
0   & 324    &324 \\
 1  & 216  &108 \\
 2& 108 & 36 \\
 3 &60& 12 \\
 4 & 16 & 4 \\
 5 & 4 & 1.33 \\ \bottomrule
\end{tabular} \\

\centering
\# of $(q-1)$-primitive normal \\ elements = 290
\end{minipage} \vspace{5mm} \quad \quad \vspace{7mm}
    &
\begin{minipage}[t][5cm]{.5\linewidth}
    \vspace{0.5mm}
\begin{tabular}{@{}|c|c|c|@{}}
\multicolumn{3}{c}{{{\normalsize Table 6: $\f_{15625}/\f_5$ ($q=5$, $m=6$)}}} \vspace{0.9mm} \\ 
\toprule
k & \# of $k$-normal elements & $\displaystyle \dfrac{\Phi_q(x^m-1)}{q^k}$ \\ \midrule
 0   & 9216    &9216 \\
 1  & 4608  & 1843.20 \\
 2& 1344 & 368.64\\
 3 &384 & 73.73 \\
 4 & 64 & 14.75\\
 5 & 8 & 2.95\\ \bottomrule
\end{tabular} \\
\vspace{0.5mm}

\centering
\# of $(q-1)$-primitive normal \\ elements =  642
\end{minipage}
    \\ 
\begin{minipage}[t][3cm]{.5\linewidth}
\begin{tabular}{@{}|c|c|c|@{}}
\multicolumn{3}{c}{{{\normalsize Table 7: $\f_{4913}/\f_{17}$ ($q=17$, $m=3$)}}} \vspace{0.9mm} \\ 
\toprule
k & \# of $k$-normal elements & $\displaystyle \dfrac{\Phi_q(x^m-1)}{q^k}$  \\ \midrule
0   & 4608    & 4608 \\
 1  & 288  & 271.06 \\
 2& 16 & 15.94 \\ \bottomrule
\end{tabular} \\
\vspace{0.5mm}

\centering
\# of $(q-1)$-primitive normal \\ elements = 288
\end{minipage} \vspace{10mm}\quad \quad \vspace{10mm}
    &
\begin{minipage}[t][3cm]{.5\linewidth}
\begin{tabular}{@{}|c|c|c|@{}}
\multicolumn{3}{c}{{{\normalsize Table 8: $\f_{2401}/\f_{7}$ ($q=7$, $m=4$)}}} \vspace{0.9mm} \\ 
\toprule
k & \# of $k$-normal elements & $\displaystyle \dfrac{\Phi_q(x^m-1)}{q^k}$ \\ \midrule
0   & 1728    & 1728\\
 1  & 576  & 246.86 \\
 2& 84 & 35.26 \\
 3& 16 & 5.04 \\ \bottomrule
\end{tabular} \\
\vspace{0.5mm}

\centering
\# of $(q-1)$-primitive normal \\ elements =  112
\end{minipage} \\
\end{longtable}

\section{Conclusions and Open Problems}

In this paper, we dealt with the recently introduced concept of $k$-normal elements in finite fields \cite{huczynska2013existence}. The existence and cardinalities of $k$-normal elements in $\fqm$ are both strongly tied to the factorization of the polynomial $x^m-1$ over $\fq$, which, in turn, depends on the factorization of cyclotomic polynomials. One does not have an explicit formula for the irreducible factors of cyclotomic polynomials, or of their degrees, and so it is not possible to directly infer the existence or numbers of $k$-normal elements. However, one may deduce several key results by forcing certain conditions on $m$, $k$, and $q$. In Theorem \ref{existencecondition}, we used the general factorization of $x^m-1$ into cyclotomic polynomials to obtain a new existence condition for $k$-normal elements.

The structure of $\fqm$ as an additive module over $\fq[x]$ plays a key role in proofs related to normal and $k$-normal bases. In Theorem \ref{lowerbound}, we furnished a lower bound for the number of $k$-normal elements in $\fqm$ under the sole assumption that at least one of them exists. The proof is inspired by the observation in \cite{hyde2018normal} that the additive module structure of $\fqm$ in fact gives rise to a group action on all the normal elements. Our bound does not require a specific form for $m$ or $q$, and therefore extends beyond the formulas provided in \cite{saygi2019number}. Two interesting problems arise in this direction.

\begin{problem} Given a $k$-normal element $\alpha$, which subsets of $\{\alpha, \alpha^q, \alpha^{q^2}, \ldots, \alpha^{q^{m-1}}\}$ with size $m-k$ or smaller, apart from $\{\alpha, \alpha^q, \alpha^{q^2}, \ldots, \alpha^{q^{m-k-1}}\}$ are linearly independent? Computer experiments show that in many cases, there do exist linearly dependent subsets with size smaller than $m-k$. \end{problem}

\begin{problem} Given a $k$-normal element $\alpha$, does there exist another $k$-normal element outside $\spann_\fq\{\alpha, \alpha^q, \alpha^{q^2}, \ldots, \alpha^{q^{m-1}}\}$? We have noted in Remark \ref{remarkref} that Theorem \ref{number_knormal} proves that the number of $k$-normal elements in this subspace is larger than $\frac{\Phi_q(x^m-1)}{q^k}$. It would also be interesting to see whether a better bound for the total number can be obtained by bounding above the intersection of the $\fq$- spans of two distinct $k$-normal elements. 
 \end{problem}

 \begin{problem} Under what circumstances is the group action \eqref{gpaxnused} free? Under what circumstances is it transitive? \end{problem}


After the proof of the well-known Primitive Normal Basis Theorem by Lenstra and Schoof \cite{lenstra1987primitive}, several interesting generalizations have been proposed. The existence and numbers of elements with different pairs of additive orders (as in Definition \ref{order}) and multiplicative group orders have been investigated by several authors. Some solved and unsolved problems in this domain may be found in \cite{huczynska2013existence}, \cite{strongprim}, \cite{mullen2016some}, and \cite{kapetanakisvariations}. We state one such relevant open problem below.

\begin{problem}[{{\cite[Problem 6.4]{huczynska2013existence}}}]\label{probref1}
Determine the existence of high-order $k$-normal elements $\alpha \in \fqm$ over $\fq$, where “high order” means $ord(\alpha) = N$, with $N$ a large positive divisor of $q^m - 1$.
\end{problem}


With Theorem \ref{existencecondition2} we answered a special case of Problem \ref{probref1}. Following the method of Lenstra and Schoof \cite{lenstra1987primitive}, we provided an existence condition for elements in $\fqm$ with maximal additive order (i.e. normal elements) that simultaneously have a non-maximal but high multiplicative order, namely $\frac{q^m-1}{q-1}$.

\section*{Acknowledgement}{This work was partially supported by Swiss National Science Foundation grant no. 188430. The authors are also greatly thankful to Gianira Alfarano for her thorough proofreading and constructive feedback on this manuscript.}
%
%
\bibliography{waifi_ref}

\bibliographystyle{plain}   

\end{document}